\documentclass[11pt,reqno]{amsart}
\usepackage[utf8]{inputenc}
\usepackage[margin=1in]{geometry}

\usepackage{amsmath, amssymb, amsthm}
\usepackage{mathrsfs}
\usepackage{wasysym}
\usepackage{textcomp,gensymb}
\usepackage{mathtools}
\usepackage{cancel}

\usepackage{xspace}
\usepackage{xifthen}
\usepackage{setspace}

\usepackage{hyperref}
\hypersetup{colorlinks=true, citecolor=blue, linkcolor=blue, urlcolor=blue, pdfstartview=FitH}
\usepackage{cleveref}
\usepackage{enumerate}
\usepackage{array}

\usepackage[dvipsnames]{xcolor}
\usepackage{ytableau}
\usepackage{soul}
\usepackage{mdframed}
\usepackage{tcolorbox}
\tcbuselibrary{theorems}
\mdfsetup{everyline=true}
\usepackage{framed}

\usepackage{tikz}
\usepackage{tikz-cd}
\usetikzlibrary{shapes.geometric}
\usetikzlibrary{decorations.pathmorphing}
\usepackage[all,2cell,cmtip]{xy}
\UseTwocells
\usepackage{tikz-qtree}
\usepackage{forest}
\usepackage{diagbox}

\usepackage{multicol}
\usepackage{bm}
\usepackage{bbm}

\usepackage{graphicx}
\usepackage{pdfpages}

\usepackage[colorinlistoftodos,color=red!60]{todonotes}

\setul{0.5ex}{0.3ex}
\setulcolor{Plum}

\usepackage{pifont}% http://ctan.org/pkg/pifont

\newtheorem{proposition}{Proposition}

\newcommand{\abs}[1]{\left\lvert#1\right\rvert}

\renewcommand*{\Re}{\operatorname{Re}}

\newcommand{\ZZ}{\mathbb{Z}}
\newcommand{\Q}{\mathbb{Q}}
\newcommand{\Qbar}{\overline{\Q}}
\theoremstyle{plain} %italicizes text
\newtheorem{theorem}{Theorem}[section]
\newtheorem*{theorem*}{Theorem}

\newtheorem{lemma}[theorem]{Lemma}

\newtheorem{conjecture}[theorem]{Conjecture}
\newtheorem*{conjecture*}{Conjecture}

\newtheorem{definition}[theorem]{Definition}

\theoremstyle{remark}

\newtheorem*{remark}{Remark}

\usepackage{float}
\usepackage[maxnames=10,minnames=10,style=numeric]{biblatex}
\renewbibmacro{in:}{}
\usepackage{enumitem}

\usepackage{caption}
\title{Parts in $k$-indivisible Partitions Always Display Biases between Residue Classes}

\author[F. Jackson]{Faye Jackson}
\address{University of Chicago, Eckhart Hall, 5734 S University Ave, Chicago, IL 60637}
\email{alephnil@uchicago.edu}
\author[M. Otgonbayar]{Misheel Otgonbayar}
\address{Massachusetts Institute of Technology, 77 Massachusetts Avenue
Cambridge, MA 02139-4307}
\email{misheel@mit.edu}


\date{\today}

\makeatletter
\providecommand\@dotsep{5}
\def\listtodoname{TODOS: Not Done (Red), Cite (Orange), Blue (Notation)}
\def\listoftodos{\@starttoc{tdo}\listtodoname}
\makeatother

\newcommand{\linecomment}[1]{}

\newcommand{\CC}{\mathbb{C}}
\DeclareMathOperator{\Gal}{Gal}

\newcommand{\Mod}[1]{\ (\mathrm{mod}\ #1)}

\numberwithin{equation}{section}

\newcommand{\Qab}{\Q^{\text{ab}}}

\crefname{equation}{equation}{equations}
\crefname{prop}{proposition}{propositions}
\Crefname{prop}{Proposition}{Propositions}

\begin{document}

\maketitle

\begin{abstract}
    Let $k, t$ be coprime integers, and let $1 \leq r \leq t$. We let $D_k^\times(r,t;n)$ denote the total number of parts among all $k$-indivisible partitions (i.e., those partitions where no part is divisible by $k$) of $n$ which are congruent to $r$ modulo $t$. In previous work of the authors \cite{kindivis}, an asymptotic estimate for $D_k^\times(r,t;n)$ was shown to exhibit unpredictable biases between congruence classes. In the present paper, we confirm our earlier conjecture in \cite{kindivis} that there are no ``ties'' (i.e., equalities) in this asymptotic for different congruence classes. To obtain this result, we reframe this question in terms of $L$-functions, and we then employ a nonvanishing result due to Baker, Birch, and Wirsing \cite{nonvanishing} to conclude that there is always a bias towards one congruence class or another modulo $t$ among all parts in $k$-indivisible partitions of $n$ as $n$ becomes large.
\end{abstract}

\section{Introduction}

A $k$-indivisible partition of some integer $n > 0$ is a nonincreasing sequence $\lambda = (\lambda_1,\ldots,\lambda_m)$ of positive integers such that $k \nmid \lambda_j$ for all $j$ and $\sum_{j=1}^m \lambda_j = n$. We write $\mathcal{D}^\times_k(n)$ for the set of all such $k$-indivisible partitions of $n$. In previous work (see \cite{kindivis}), the authors studied the number of parts congruent to $r$ with respect to a fixed modulus $t$ among all $k$-indivisible partitions of $n$. Formally, this quantity can be defined as
\[
    D_k^\times(r,t;n) \coloneqq \sum_{\lambda \in \mathcal{D}^\times_k(n)} \#\{\lambda_j \mid \lambda_j \equiv r \Mod t\}.
\]
Using Wright's circle method, the authors proved the following asymptotic estimate for $D_k^\times(r,t;n)$ as $n \to \infty$ when $k,t \geq 2$ are taken to be coprime and $1 \leq r \leq t$ (see \cite[Theorem~1.1]{kindivis}):
\begin{equation}
        D_k^\times(r,t;n) = A_{k,t}(n)\left(\frac{K}{2}\log n + \left(- \psi\left(\frac{r}{t}\right) +k^{-1}\psi\left(\frac{\rho_{k,t}(r)}{t}\right)\right) + C_{k,t} + O\left(n^{-\frac{1}{2}}\log n\right)\right), \label{eq:kindivis-asym}
\end{equation}
where $\psi(x) \coloneqq \frac{\Gamma'(x)}{\Gamma(x)}$ is the digamma function, $1 \leq \rho_{k,t}(r) \leq t$ is a representative\footnote{In \cite{kindivis}, this representative $\rho_{k,t}(r)$ is denoted $\overline{r}$, suppressing the dependence on $k,t$. In this paper, we maintain this dependence due to its appearance in later proofs} of $k^{-1}r$ modulo $t$, and
\begin{align*}
    K \coloneqq 1 - \frac{1}{k} && C_{k,t} \coloneqq \frac{K}{2}\log\left(\pi\sqrt{\frac{K}{6}}\right) - K \log t + \frac{\log k}{k} && A_{k,t}(n) \coloneqq \frac{3^{\frac{1}{4}}e^{\pi\sqrt{\frac{2Kn}{3}}}}{2^{\frac{3}{4}}K^{\frac{1}{4}}n^{\frac{1}{4}}\pi t\sqrt{k}}.
\end{align*}
This asymptotic implies a weak asymptotic equidistribution among the congruence classes modulo $t$, i.e. for each $1 \leq r,s \leq t$ we have that $\frac{D_k^\times(r,t;n)}{D_k^\times(s,t;n)} \to 1$ as $n \to \infty$. However, it also implies a bias towards certain congruence classes. If we define
\[
    \psi_{k,t}(r) \coloneqq -\psi\left(\frac{r}{t}\right) + \frac{1}{k}\psi\left(\frac{\rho_{k,t}(r)}{t}\right),
\]
then if $\psi_{k,t}(r) < \psi_{k,t}(s)$, we have that $D_k^\times(r,t;n) < D_k^\times(s,t;n)$ for large $n$. This encourages defining an ordering $\prec_{k,t}$ on the integers $\{1,\ldots,t\}$ (equivalently on $\ZZ/t\ZZ$), where $r \prec_{k,t} s$ provided that $D_k^\times(r,t;n) < D_k^\times(s,t;n)$ for sufficiently large $n$. This ordering depends simultaneously on the size of $k$ and the congruence class of $k$ modulo $t$, and is incredibly intricate. For a detailed exposition of the known properties of $\prec_{k,t}$, we refer to the authors' previous paper \cite{kindivis}. It is unclear from \cref{eq:kindivis-asym} that $\prec_{k,t}$ is necessarily a total ordering, and so the following conjecture was posed within the author's previous paper.
\begin{conjecture*}[No Ties {\cite[Conjecture~1.3]{kindivis}}]
    The ordering $\prec_{k,t}$ is a total ordering on $\ZZ/t\ZZ$ when $k,t \geq 2$ are coprime.
\end{conjecture*}
This conjecture is resolved by the following theorem, which constitutes our principal result in the current paper.
\begin{theorem}\label{thm:no-ties}
    Let $k,t \geq 2$ be coprime, then for any $1 \leq r \neq s \leq t$, we have that $\psi_{k,t}(r) \neq \psi_{k,t}(s)$. As a consequence, the No Ties Conjecture holds.
\end{theorem}

\begin{remark}
    In general, the set $\{\psi_{k,t}(r) \mid 1 \leq r \leq t\}$ is not linearly independent over $\Q$, even for $k,t$ coprime. This can be seen via the following example, from \cite{murty}:
    \begin{align*}
        \psi(1/4) - 3\psi(1/2) + \psi(3/4) + \psi(1) &= 0.
    \end{align*}
    Scaling by $\frac{1}{1- 1/k}$ for any $k$ coprime to $4$ yields
    \[
        \psi_{k,4}(1) - 3\psi_{k,4}(2) + \psi_{k,4}(3) + \psi_{k,4}(4) = 0,
    \]
    since $k^{-1}r \equiv r$ modulo $4$ for any $1 \leq r \leq 4$.
\end{remark}
In qualitative terms, \Cref{thm:no-ties} states that there are no  ``ties'' (i.e., equalities) between different congruence classes in the second-order term of our asymptotic. The techniques we use to prove \Cref{thm:no-ties} are inspired by the work of Gun, Murty, and Rath concerning the linear independence of the set $\{\psi(a/t) \mid \gcd(a,t) = 1\}$ over different number fields \cite{murty}. In particular, we make use of a connection between $L$-functions and linear combinations of values of the digamma function at rational points developed by Murty and Saradha in \cite{murtyTrans}. Furthermore, we prove a more general theorem concerning the structure of linear relations among values of the digamma function at rational points, using the same techniques.

\begin{theorem}\label{thm:const-lin-comb}
    Let $K$ be some finite extension of $\Q$ over which the $t$-th cyclotomic polynomial is irreducible. For a function $f : \ZZ/t\ZZ \to K$, let $A_f \coloneqq \{\gcd(n,t) \mid f(n) \neq 0\}$. Now let $f : \ZZ/t\ZZ \to K$ be some nonzero function such that
    \begin{align}
        \sum_{n=1}^t f(n)\psi(n/t) = 0 && \sum_{n=1}^t f(n) = 0.\label{eq:f-vanishes}
    \end{align}
    Then we have that $A_f \nsubseteq \{a,t\}$ for any $a$ dividing $t$.
    
    Furthermore, if $A_f$ is minimal with respect to inclusion among functions satisfying \cref{eq:f-vanishes}, then $f(n) = f(hn)$ for all $n$ and for any $h \in (\ZZ/t\ZZ)^\times$ satisfying $h \equiv 1 \pmod{t/a}$ for any $a \in A_f$.
\end{theorem}
By translating these questions into the setting of nonvanishing of certain $L$-functions, we are able to apply arguments of Baker, Birsch, and Wirsing from \cite{nonvanishing} in order to reduce \Cref{thm:no-ties,thm:const-lin-comb} to simple facts about the action of $(\ZZ/t\ZZ)^\times$ (i.e., the group of units in $\ZZ/t\ZZ$) on $\ZZ/t\ZZ$. 

\begin{remark}
    \Cref{thm:no-ties} follows via a short argument from \Cref{thm:const-lin-comb}. However, for ease of exposition, we prove \Cref{thm:no-ties} independently, as here the argument can be framed more cleanly using the language of equivalence relations.
\end{remark}

For the sake of convenience, we will call an equality $\psi_{k,t}(r) = \psi_{k,t}(s)$ a ``tie'' between $r,s$ modulo $t$ in $k$-indivisible partitions. Our goal is to show that any such tie implies that $r = s$. The paper is organized as follows. In \Cref{sec:ties-propagate} we detail the relationship between $L$-functions and ties, using it to prove key properties that ties must satisfy if they exist. In \Cref{sec:no-ties}, we use these properties to prove \Cref{thm:no-ties}. Using the same techniques, we prove \Cref{thm:const-lin-comb} in \Cref{sec:lin-comb}. We also rephrase Baker, Birch, and Wirsing's fundamental lemma in terms of the modern theory of Galois representations in \Cref{sec:galois}. Finally, we suggest future problems concerning the behavior of $\psi_{k,t}(r)$ in \Cref{sec:future}.

\section*{Acknowledgements}

The authors were participants in the 2022 UVA REU in Number Theory. They would like to thank Ken Ono, the director of the UVA REU in Number Theory, their graduate student mentor William Craig, and Steven J. Miller for their support and helpful comments. They would also like to thank their colleagues at the UVA REU for their encouragement and support. They are grateful for the support of grants from the National Science Foundation (DMS-2002265, DMS-2055118, DMS-2147273), the National Security Agency (H98230-22-1-0020), and the Templeton World Charity Foundation.

\section*{Data Availability}

The authors implemented a program in Mathematica to generate the figures found in \Cref{sec:future}. This program can be obtained from GitHub at
\begin{center}
    \href{https://github.com/FayeAlephNil/KRegularBiases}{https://github.com/FayeAlephNil/KRegularBiases}
\end{center}
or, upon reasonable request, from the authors.

\section{Propagation of Ties}\label{sec:ties-propagate}

We wish to show that there are no ties in the second order term of our asymptotic.
Before we begin, we will rephrase the question into one about equivalence relations on $\ZZ/t\ZZ$, as this notation will be quite useful for us.
\begin{definition}
    Let $k,t$ be coprime. For any $1 \leq r,s \leq t$, we say that $r \sim_{k,t} s$ provided that
    \begin{align*}
        \psi_{k,t}(r) = \psi_{k,t}(s).
    \end{align*}
    We may view $r,s$ as elements of $\ZZ/t\ZZ$, and thus this equivalence relation as a relation on $\ZZ/t\ZZ$. We use the notation $[r]_{k,t}$ for the equivalence class of $r$.
\end{definition}
Our goal is to show that $\sim_{k,t}$ is trivial, that is, $r \sim_{k,t} s$ implies $r=s$. To show \Cref{thm:no-ties}, we must first explore the properties of $\sim_{k,t}$. From our previous work in \cite{kindivis}, we already know that $[1]_{k,t} = \{1\}$. Our primary technique will be to reduce to this fact; namely, we will be able to derive the theorem from the following three properties of $\sim_{k,t}$.
\begin{lemma}\label{lemma:ties-properties}
    Let $k,t$ be coprime and $1 \leq r,s \leq t$, then we have the following.
    \begin{enumerate}
        \item\label{item:common-div} If $r,s,t$ share some common factor $x$, then $\frac{r}{x} \sim_{k,t/x} \frac{s}{x}$.
        \item\label{item:1-common} If $1 \sim_{k,t} r$, then $r = 1$.
        \item\label{item:twist} If $r \sim_{k,t} s$ and $h$ is coprime to $t$, then $rh \sim_{k,t} sh$.
    \end{enumerate}
\end{lemma}
The first property follows from the fact that $\psi_{k,t}(r) = \psi_{k,t/x}(r/x)$ when $x$ divides both $r,t$. This is an immediate consequence of computing that $\rho_{k,t}(r) = \rho_{k,t/x}(r/x)$, which can be shown by elementary modular arithmetic (recall that $1 \leq \rho_{k,t}(r) \leq r$ is a representative of $k^{-1}r$ modulo $t$). The second property was proved in \cite{kindivis}, as noted before. To prove the third property, we will use the connection to $L$-functions developed by Murty and Saradha in \cite{murtyTrans}.

\subsection{Ties and the Vanishing of \texorpdfstring{$L$}{L}-functions}

To understand the existence of ties, we first use a result of Murty and Saradha (see \cite{murtyTrans}) to relate this to the nonvanishing of a particular Dirichlet-like $L$-function. Then, we will apply a key lemma appearing in the proof of a nonvanishing theorem due to Baker, Birch, and Wirsing concerning such $L$-functions (see \cite{nonvanishing}). Our exposition of the following connection follows that of Murty, Gun and Rath in \cite{murty}.

To begin, let $f$ be some periodic arithmetic function with period $t$. Throughout, $f$ takes algebraic values, and in fact in our application the values of $f$ will be rational. To this function, we can associate an $L$-series
\[
    L(s,f) = \sum_{n=1}^\infty \frac{f(n)}{n^s},
\]
which converges for $\Re s > 1$. It is well known that $L(s,f)$ may analytically continued to the entire complex plane, apart from $s=1$, where $L(s,f)$ may have a simple pole. The residue at this pole is given by $\sum_{r=1}^t f(r)$. Furthermore, if this residue is zero, then the series
\begin{align*}
    \sum_{n=1}^\infty \frac{f(n)}{n}
\end{align*}
converges to $L(1,f)$. As shown by Murty and Saradha \cite{murtyTrans} this can be related to sums of the digamma function at rational arguments, since in this case
\begin{align}
    L(1,f) = -\frac{1}{t}\sum_{r=1}^t f(r)\psi(r/t).\label{eq:murtyTrans}
\end{align}
This provides our candidate choice of a function $f$. Let $\mathbbm{1}_r : \ZZ/t\ZZ \to \{0,1\}$ be the indicator function for the congruence class $r$ modulo $t$. Define
\begin{align}
    f_{r,s}(n) &\coloneqq -\mathbbm{1}_r(n) + \mathbbm{1}_s(n) + \frac{\mathbbm{1}_{k^{-1}r}(n)}{k} - \frac{\mathbbm{1}_{k^{-1}s}(n)}{k}. \label{eq:f-rs}
\end{align}
Then by \cref{eq:murtyTrans}, $L(1,f_{r,s}) = 0$ if and only if $r \sim_{k,t} s$, as 
\[
    L(1,f) = -\frac{1}{t}\left(\psi_{k,t}(r) - \psi_{k,t}(s)\right).
\]
As Baker, Birch, and Wirsing's result concerns series of the form $\sum_{n=1}^\infty \frac{f(n)}{n}$, we've thus converted our problem into these terms. Using their notation, let $F_t$ be the collection of such functions $f$ which are periodic with period $t$ and taking algebraic values for which
\[
    \sum_{n=1}^\infty \frac{f(n)}{n} = 0.
\]
For convenience, we also let $\Qbar$ denote the algebraic numbers and $\xi = e^{2\pi i/t}$. We are now prepared to state the key lemma.
\begin{lemma}[Baker, Birch, Wirsing {\cite[Lemma~4]{nonvanishing}}]\label{lemma:nonvanishing}
    Let $f$ be periodic with period $t$ and suppose $L(1,f) = 0$. Furthermore, let $\sigma$ be any automorphism of $\Qbar$, and let $h$ the integer defined modulo $t$ by $\sigma^{-1}\xi = \xi^h$. Then $f^\sigma(n) \coloneqq \sigma f(hn)$ also satisfies $L(1,f^\sigma) = 0$.
\end{lemma}

Equipped with this lemma, we are prepared to prove \cref{item:twist} in \Cref{lemma:ties-properties}.
\begin{proof}[Proof of \cref{item:twist} in \Cref{lemma:ties-properties}]
    Let $r,s \in \ZZ/t\ZZ$ and let $f$ be defined by \cref{eq:f-rs}. Furthermore, fix some $h$ coprime to $t$. By coprimality, we know there is some automorphism $\sigma$ of $\overline{\mathbb{Q}}$ so that $\sigma^{-1}\xi = \xi^h$. Furthermore, since $f$ takes rational values, we see that
    \begin{align*}
        f^\sigma(n) = \sigma f(hn) = f(hn).
    \end{align*}
    We now may compute that, as a function on $\ZZ/t\ZZ$,
    \begin{align*}
        f^\sigma(n) &= -\mathbbm{1}_r(hn) + \mathbbm{1}_s(hn) + \frac{\mathbbm{1}_{k^{-1}r}(hn)}{k} - \frac{\mathbbm{1}_{k^{-1}s}(hn)}{k} \\
        &= -\mathbbm{1}_{h^{-1}r}(n) + \mathbbm{1}_{h^{-1}s}(n) + \frac{\mathbbm{1}_{k^{-1}h^{-1}r}(n)}{k} - \frac{\mathbbm{1}_{k^{-1}h^{-1}s}(n)}{k}.
    \end{align*}
    Since we know $r \sim_{k,t} s$ if and only if $L(1,f) = 0$, this implies $L(1,f^\sigma) =0$ by \Cref{lemma:nonvanishing}. This in turn implies $h^{-1}r \sim_{k,t} h^{-1}s$. Because $(\ZZ/t\ZZ)^\times$ is a group, this is sufficient.
\end{proof}

\section{Proof of \texorpdfstring{\Cref{thm:no-ties}}{Theorem 1.2}}\label{sec:no-ties}

Equipped with the above propagation results, we may begin our proof that $\sim_{k,t}$ is trivial. We will only use the properties of $\sim_{k,t}$ listed in \Cref{lemma:ties-properties}, which we recall here.
\begin{enumerate}
    \item If $x \mid r,s,t$ and $r \sim_{k,t} s$ then $\frac{r}{x} \sim_{k,t/x} \frac{s}{x}$.
    \item $1 \sim_{k,t} s$ implies $s = 1$.
    \item If $h$ is coprime to $t$ and $r \sim_{k,t} s$ then $rh \sim_{k,t} sh$.
\end{enumerate}
We begin with the simplest case, when one of $r,s$ is coprime to $t$.
\begin{lemma}\label{lemma:coprime-no-ties}
    If $r \sim_{k,t} s$ and one of $r,s$ is coprime to $t$, then $r=s$. Furthermore, if $r \sim_{k,t} s$ and $\gcd(r,t) = \gcd(s,t)$ then $r = s$.
\end{lemma}

\begin{proof}
    Without loss of generality, suppose that $r$ is coprime to $t$. Then by \cref{item:twist}, $1 \sim_{k,t} r^{-1}s$, and so $s = r$ since $[1]_{k,t} = \{1\}$. The latter statement follows by \cref{item:common-div}, by reducing with respect to $\gcd(r,t) = \gcd(s,t)$.
\end{proof}
We now
% begin in earnest.
handle the general case.
For the sake of contradiction, take $r \neq s$ within $\ZZ/t\ZZ$ satisfying $r \sim_{k,t} s$.
Furthermore, take $t$ to be the smallest $t$ where this occurs.
For convenience, let $a = \gcd(r,t) > 1, b = \gcd(s,t) > 1$ and $r = ar', s = bs'$
($a = 1$ or $b = 1$ is the coprime case above).
If $x \mid a,b$ and $x > 1$, then we may reduce $r,s,t$ by $x$ via \cref{item:common-div} to get a tie with smaller $t$.
% Thus, by infinite descent, we may take $\gcd(a,b) = 1$.
Since we took the smallest $t$, this is a contradiction. Thus we must have $\gcd(a, b) = 1$.
At least one of $a,b > 1$ is odd, and so without loss of generality, let $p \mid a$, with $p > 2$.

Our
% primary goal
aim here
is to build some $h$ coprime to $t$ so that $rh \equiv r \pmod t$, and $sh \not\equiv s \pmod t$. We will then have that $sh \sim_{k,t} s$, but $\gcd(sh,t) = \gcd(s,t)$ since $h$ is coprime to $t$. This contradicts \Cref{lemma:coprime-no-ties}. As this fact is purely elementary number theory, we separate it into its own lemma.
\begin{lemma}
    Let $r,s \in \ZZ/t\ZZ$ with $\gcd(r,t), \gcd(s,t) > 1$ be coprime. Furthermore, let $p > 2$ be some prime so that $p \mid \gcd(r,t)$. Then there is some $h$ coprime to $t$ so that $rh \equiv r \pmod t$, but $sh \not\equiv s \pmod t$.
\end{lemma}

\begin{proof}
    We set
    \begin{align*}
        h \coloneqq \frac{ty}{p} + 1,
    \end{align*}
    where $y$ is some integer which is nonzero modulo $p$. We claim we can choose $y$ modulo $p$ so that $h$ satisfies the desired properties. First we see that
    \begin{align*}
        rh &= \frac{rty}{p} + r \equiv r \pmod t \\
        sh &= \frac{sty}{p} + s \not\equiv s \pmod t.
    \end{align*}
    The first statement follows since $p \mid r$, so $t \mid rty/p$. The second follows since if $t \mid sty/p$, we have that $p \mid sy$, but if $p \mid sy$ then $p \mid \gcd(s,t)$ or $p \mid y$. In the first case, $\gcd(r,t)$ and $\gcd(s,t)$ are not coprime, and in the second, $y \equiv 0 \pmod p$.
    
    Now we must show that $y$ may be chosen such that $h$ is coprime to $t$. Let $p' \neq p$ be some prime dividing $t$. Then $h \equiv 1 \pmod{p'}$. Modulo $p$, we find that $h \equiv \frac{ty}{p} + 1$, thus it suffices to pick some $y \not\equiv 0 \pmod p$ so that $ty/p \not\equiv -1 \pmod p$. There is at most one $y$ modulo $p$ so that $ty/p \equiv -1 \pmod p$, and there are $p-1 > 1$ nonzero elements of $\ZZ/p\ZZ$. Thus, pick a $y \not\equiv 0 \pmod p$ with $ty/p \not\equiv -1 \pmod p$. Such a $y$ has that $h$ is not divisible by any prime dividing $t$, and so $h,t$ are coprime as desired.
\end{proof}
With this, we have \Cref{thm:no-ties} by the argument above.\qed

\section{Proof of \texorpdfstring{\Cref{thm:const-lin-comb}}{Theorem 1.2}}\label{sec:lin-comb}

Fix some finite extension $K$ of $\Q$ over which the $t$-th cyclotomic polynomial is irreducible. Furthermore, fix some function $f : \ZZ/t\ZZ \to K$ satisfying
\begin{align}
    \sum_{n=1}^t f(n)\psi(n/t) = 0 && \sum_{n=1}^t f(n) = 0.\label{eq:f-vanishes-2}
\end{align}
Via the connection to $L$-functions, this is equivalent to the condition that
\[
    L(1,f) = \sum_{n=1}^\infty \frac{f(n)}{n} = 0.
\]
As before, let $\xi_t \coloneqq e^{2\pi i/t}$ be a primitive $t$-th root of unity. By \cite[Theorem~1]{nonvanishing}, we cannot have $A_f \subseteq \{1,t\}$. Similarly, we have the following more general result.
\begin{lemma}
    $A_f \nsubseteq \{a,t\}$ for any $a \in \ZZ/t\ZZ$.
\end{lemma}

\begin{proof}
    Suppose $A_f \subseteq \{a,t\}$. We see that the linear combination above may be rewritten as
    \begin{align*}
        \sum_{\substack{n=1 \\ \gcd(n,t) = a,t}}^t f(n)\psi\left(\frac{n/a}{t/a}\right) = 0.
    \end{align*}
    Setting $g(m) = f(ma)$, then $g : \ZZ/(t/a)\ZZ \to K$, and furthermore $L(1,g) = 0$, as
    \begin{align*}
        \sum_{m=1}^{t/a} g(n) \psi(m/t)  = 0 && \sum_{m=1}^{t/a} g(n) = 0.
    \end{align*}
    Thus by \cite[Theorem~1]{nonvanishing}, we have a contradiction, as $A_g \subseteq \{1,t/a\}$, and clearly the $t/a$-th cyclotomic polynomial is irreducible over $K$.
\end{proof}
This also shows that $\abs{A_f} \geq 2$. Now suppose that $A_f$ is minimal with respect to subset inclusion among nonzero functions satisfying \cref{eq:f-vanishes-2}.

Fix some $r$ with $f(r) \neq 0$ and let $a \coloneqq \gcd(r,t)$. We must show the equality $f(n) = f(hn)$ for any $n$ and any $h \in (\ZZ/t\ZZ)^\times$ satisfying $h \equiv 1 \pmod{t/a}$. First we see from \Cref{lemma:nonvanishing} that
\begin{align*}
    \sum_{n=1}^t f(hn)\psi(n/t) = 0,
\end{align*}
since there is some $\sigma \in \Gal(\Qbar/K)$ with $\sigma^{-1}\xi_t = \xi_t^h$ as the $t$-th cyclotomic polynomial is irreducible over $K$. This implies that $L(1,f^\sigma) = 0$, which then implies the claimed equality. Now set $g(n) = f(n) - f(hn)$. We see that
\begin{align*}
    \sum_{n=1}^t g(n)\psi(n/t) = 0 && \sum_{n=1}^t g(n) = 0.
\end{align*}
We will now show that $A_g \subsetneq A_f$, and thus $g(n) = 0$ for all $n$, proving the theorem.

First note that $A_g \subseteq A_f$, since $\gcd(hn,t) = \gcd(n,t)$. Thus it suffices to show that $a \not\in A_g$. To do this, let $s \in \ZZ/t\ZZ$ with $\gcd(s,t) = a$. We first compute $hs$ modulo $t$. To do this note that since $a \mid s$ and also $t/a \mid h-1$ we have that that $t \mid (h-1)s$, and thus $hs = s \pmod{t}$. This completes the proof, as then $g(s) = 0$, for all such $s$, showing that $A_g \subsetneq A_f$. \qed

\section{Connection to Galois Representations}\label{sec:galois}

The key lemma of Baker, Birch, and Wirsing can be phrased in terms of a $K$-linear representation of the Galois group $\Gal(\Qbar/K)$. This illustrates an unexpected relationship between linear relations among values of the digamma function at rational values and representation theory. Throughout this section, we'll consider the following $\Qbar$-vector space $V$:
\begin{align*}
    \mathcal{F} \coloneqq \left\{f : \ZZ \to \Qbar \mid f \text{ is periodic}\right\}.
\end{align*}
Note that $\mathcal{F}$ is a $\Qbar$ vector space with countable dimension. This vector space is equipped with linear maps $L(s,-)$ from $\mathcal{F}$ into $\CC$, given by the $L$-function
\begin{align*}
    L(s,f) = \sum_{n=1}^\infty \frac{f(n)}{n^s}
\end{align*}
for $\Re(s) > 1$. Of particular interest to us are the subspaces
\begin{align*}
    \mathcal{F}^0 &\coloneqq \left\{f \in \mathcal{F} \mid \sum_{n=1}^t f(n) = 0, \text{ where } t \text{ is the period of } f\right\} \\
    \mathcal{F}_{K} &\coloneqq \left\{f \in \mathcal{F} \mid \forall n \in \ZZ, f(n) \in K\right\} \\
    \mathcal{F}_{K,t} &\coloneqq \left\{f \in \mathcal{F}_K \mid f \text{ has period } t\right\},
\end{align*}
as well as $\mathcal{F}_K^0, \mathcal{F}_{K,t}^0$ (defined in the same way as $\mathcal{F}^0$). On $\mathcal{F}^0$, the $L$-functions above may be extended to $s = 1$, and so
\begin{align*}
    L(1,f) = \sum_{n=1}^\infty \frac{f(n)}{n}
\end{align*}
defines a linear operator on $\mathcal{F}^0$. In the previous section, we were focused on understanding whether certain linear combinations belonged to the kernel of $L(1,-)$. The key lemma of Baker, Birch,
% Birch?
and Wirsing amounts to the invariance of this kernel under a certain action of $\Gal(\Qbar/K)$.
\begin{proposition}
    Let $\sigma \in \Gal(\Qbar/K)$. For any $f \in \mathcal{F}$, define $f^\sigma$ by the formula
    \begin{align*}
        f^\sigma(n) = \sigma f(hn)
    \end{align*}
    where, if $f$ has period $t$, $h$ is defined modulo $t$ by $\sigma^{-1}\xi_t = \xi_t^h$. Then in fact, $(\sigma,f) \mapsto f^\sigma$ defines a $K$-linear representation of
    $\Gal(\Qbar/K)$.
    % wait f??? shouldn't it be Gal_Q??
    
    Furthermore, if $L \subseteq \Qbar$ is $\Gal(\Qbar/K)$ invariant then $\mathcal{F}^0,\mathcal{F}_L,\mathcal{F}_{L,t}$ are invariant subspaces for this representation. More strikingly, $\ker L(1,-)$ is an invariant subspace of $\mathcal{F}^0$.
\end{proposition}
\vspace{-0.1in}
\begin{proof}
    The claimed invariances for $\mathcal{F}^0,\mathcal{F}_L,$ and $\mathcal{F}_{L,t}$ are easy to check, and the final statement follows from the work of Baker, Birch, Wirsing in \cite[Lemma~4]{nonvanishing}. Thus it suffices to show that this mapping truly defines a $K$-linear action. First, if $f$ has period $t$, it also has period $tx$ for any natural number $x \geq 1$, so we must show that the action does not depend on choosing the minimal period. To do this, note that since $\xi_{t} = \xi_{tx}^x$
    \begin{align*}
        \sigma^{-1} \xi_t &= (\sigma^{-1} \xi_{tx})^x = (\xi_{tx}^h)^x = \xi_t^h,
    \end{align*}
    where $\sigma^{-1}\xi_{tx} = \xi_{tx}^h$.
    
    To show that this truly defines an action, fix $f,g \in \mathcal{F}$ of periods $t_1,t_2$ respectively as well as $\sigma,\tau \in \Gal(\Qbar/K)$ and $c \in K$. First note that $f+g$ has period $t = t_1t_2$, and by the above we may set $h,h'$ so that $\sigma^{-1}\xi_t = \xi_t^h, \tau^{-1}\xi_t = \xi_t^{h'}$. Then we have that
    \begin{align*}
        (cf)^\sigma(n) &= \sigma(cf(hn)) = \sigma(c)\sigma f(hn) = c\sigma f(hn) = cf^\sigma(n) \\
        (f+g)^\sigma(n) &= \sigma((f+g)(hn)) = \sigma f(hn) + \sigma g(hn) = f^\sigma(n) + g^\sigma(n),
    \end{align*}
    and also
    \begin{align*}
        (\tau\sigma)^{-1} \xi_t &= \sigma^{-1}\tau^{-1} \xi_t = \sigma^{-1}\xi_t^{h'} = \xi_t^{h'h} \\
        (f^\sigma)^{\tau}(n) &= \tau f^\sigma(h'n) = \tau\sigma f(hh'n) = f^{\tau\sigma}(n).
    \end{align*}
    Therefore this is in fact a $K$-linear representation of $\Gal(\Qbar/K)$.
\end{proof}

\begin{remark}
    Let $\Qab$ denote the maximal abelian extension of $\Q$. The action of $\Gal(\Qbar/K)$ on $\mathcal{F}_L$ where $L \subseteq K \subseteq \Qab$ is defined entirely in terms of its action on the roots of unity. Thus, this representation factors through $\Gal(\Qab/K)$. This greatly simplifies the analysis of the problem, as $\Gal(\Qab/K)$ can be computed explicitly.
    
    In our particular case, we deal with the subspaces $\mathcal{F}_{K,t}^0$ for fixed $t$ and $K \subseteq \Qab$. Here, the representation is completely determined by the action of $\Gal(\Qab/K)$ on the $t$-th roots of unity, and so in fact factors through $\Gal(\Q(\xi_t)/K) \subseteq \Gal(\Q(\xi_t)/\Q) \cong (\ZZ/t\ZZ)^\times$. Furthermore, the space $\mathcal{F}_{K,t}$ is finite-dimensional over $K$, and so the representation theory is entirely classical.
\end{remark}

\section{Future Work}\label{sec:future}

\Cref{thm:no-ties} shows that $\psi_{k,t}(r) - \psi_{k,t}(s) \neq 0$ for any coprime $k,t$ and distint $1 \leq r,s \leq t$. This leaves open the question of how large or how small this difference can be, and it also leaves open the question of whether two of these differences can coincide. More qualitatively, how quickly does the gap between the counts $D^\times_k(r,t;n)$ and $D_k^\times(s,t;n)$ grow, and does the gap grow at a different rate for each pair $r,s$? 

The methods employed in this paper are unequipped to deal with the first question outright, as the Galois twist $f \mapsto f^\sigma$ simply preserves $\ker L(1,-)$, and it is not known how the size of $L(1,f)$ compares to that of $L(1,f^\sigma)$ when $L(1,f) \neq 0$. For convenience, define the quantities
\begin{align*}
    \Psi_{k,t} &\coloneqq \{\abs{\psi_{k,t}(r) - \psi_{k,t}(s)} \mid 1 \leq r < s \leq t\} \\
    \mathcal{M}_{k,t} &\coloneqq \max \Psi_{k,t} && \mathcal{M}_t \coloneqq \max_{\gcd(k,t) = 1} \mathcal{M}_{k,t} \\
    m_{k,t} &\coloneqq \min \Psi_{k,t} && m_t \coloneqq \min_{\gcd(k,t) = 1} m_{k,t}.
\end{align*}
As we can see in \Cref{fig:M1}, $\mathcal{M}_t$ grows linearly in $t$. In fact, something better is true. The quadruples $(k,t,r,s)$ for which $\abs{\psi_{k,t}(r) - \psi_{k,t}(s)}$ are maximal are extremely predictable. They are given by $(k,t,r,s) = (t-1,t,1,t-1)$.

In contrast, $m_t$ behaves erratically. A naive plot of $m_t$ is displayed in \Cref{fig:m}. Plotting $-\log m_t$ gives a more illuminating picture, as in \Cref{fig:mlog}. Motivated by these plots, we make the following conjecture.
\begin{conjecture}\label{conj:mtMT}
    $\mathcal{M}_t/t \to 1$ as $t \to \infty$ and the function $\abs{\log m_t/\sqrt{t}}$ is bounded. Furthermore the maximum $\mathcal{M}_t$ is achieved by the quadruples $(k,t,r,s) = (t-1,t,1,t-1)$.
\end{conjecture}

\begin{remark}
    Generating these plots is very computationally expensive. Naively, one must compute $\abs{\psi_{k,t}(r)-\psi_{k,t}(s)}$ for each $1 \leq r,s \leq t$ and each $k$ coprime to $t$ and less than $\frac{6}{\pi^2}(t^2-1)$. Improvements to the algorithm used to produce these numerics would help provide further evidence of \Cref{conj:mtMT}
\end{remark}

\begin{figure}[H]
    \centering
    \includegraphics[scale=0.9]{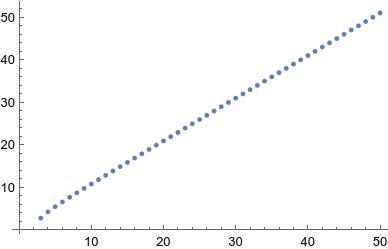}
    \caption{Plot of $\mathcal{M}_t$ for $3 \leq t \leq 50$}
    \label{fig:M1}
\end{figure}

\begin{figure}[H]
    \centering
    \includegraphics[scale=0.6]{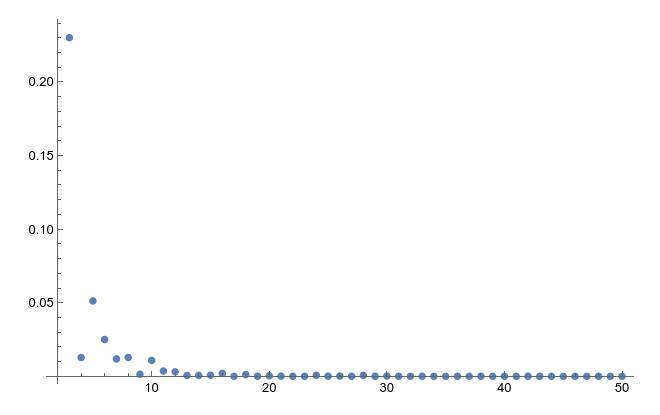}
    \caption{Plot of $m_t$ for $3 \leq t \leq 50$}
    \label{fig:m}
\end{figure}

\begin{figure}[H]
    \centering
    \includegraphics[scale=0.4]{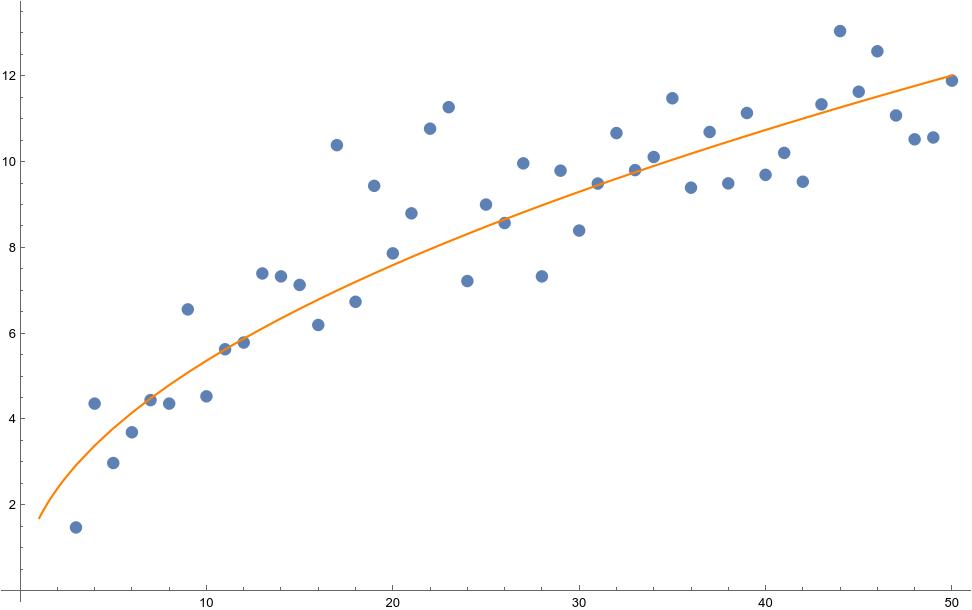}
    \caption{Plot of $-\log m_t$ in blue versus $1.7\sqrt{t}$ in orange for $3 \leq t \leq 50$}
    \label{fig:mlog}
\end{figure}

\printbibliography

\end{document}